\documentclass[12pt]{article}

\usepackage{amsmath, amsfonts, amssymb, amsthm, graphicx, color, enumitem}
\usepackage{hyperref}

\newcommand{\HilH}{\mathcal{H}}
\newcommand{\realR}{\mathbb{R}}
\newcommand{\compC}{\mathbb{C}}

\newcommand{\resp}{resp.\@}

\newcommand{\eg}{e.g.}

\newcommand{\vect}[1][t]{\mathbf{#1}}
\newcommand{\lHopital}{l'H\^{o}pital}

\newcommand{\R}{\mathbb{R}}

\newcommand{\N}{\mathbb{N}}

\newcommand{\m}{\mathbf{m}}

\newcommand{\ga}{\mathbf{\Gamma}}

\newcommand{\bfGamma}{\mathbf{\Gamma}}

\newcommand{\Andreief}{Andr\'{e}ief}

\DeclareMathOperator{\Tr}{Tr}

\newtheorem{thm}{Theorem}[section]
\newtheorem{prop}{Proposition}[section]

\theoremstyle{definition}

\theoremstyle{remark}
\newtheorem{rmk}{Remark}[section]

\newcommand{\ddd}{d}

\newcommand{\fE}{\mathfrak{E}}

\title{On a relationship between high rank cases and rank one cases of Hermitian random matrix models with external source}
\author{Jinho Baik\thanks{Department of Mathematics, University of Michigan, Ann Arbor, MI, 48109, USA \newline
email: \texttt{baik@umich.edu}} \ and 
Dong Wang\thanks{Department of Mathematics, University of Michigan, Ann Arbor, MI, 48109, USA \newline
email: \texttt{dowang@umich.edu}}}

\date{\today}

\begin{document}

\maketitle


\begin{abstract} 
We prove an identity on Hermitian random matrix models with external source relating the high rank cases to the rank $1$ cases. 
This identity was proved and used in a previous paper of ours to study the  asymptotics of the top eigenvalues. 
In this paper, we give an alternative, more conceptual proof of this identity based on a connection between the Hermitian matrix models with external source and the discrete KP hierarchy. 
This connection is obtained using the vertex operator method of Adler and van Moerbeke.
The desired identity then follows from the Fay-like identity of the discrete KP $\tau$ vector.
\end{abstract}

\maketitle

\section{Introduction} \label{section:introduction}

The subject of this paper is an identity between a Hermitian random matrix model with external source of rank $\m$ and $\m$ such models with external source of rank $1$. 
This identity allows us to reduce the asymptotic study of ``spiked Hermitian random matrix models'' of rank higher than $1$ to that of the models of rank $1$. 
This reduction formula was used in \cite{Baik-Wang10} to evaluate the limiting fluctuations of the top eigenvalue(s) of spiked models of arbitrary fixed rank with a general class of potentials. 
In \cite{Baik-Wang10} we gave a direct proof of this identity using the formula of \cite{Baik09} on the Christoffel-Darboux kernel. 
Here we give a different, more conceptual proof using the relation between the random matrix model with external source and discrete KP hierarchy. 
We show how the general results of Adler and van Moerbeke \cite{Adler-van_Moerbeke99a} 
on the construction of solutions of discrete KP hierarchy can be used on the partition functions of the Hermitian matrix model with external source.

\bigskip

We now introduce the model. Let $W(x)$ be a piecewise-continuous function on $\realR$. 
We assume that $W(x)$ is nonnegative, has infinite support, and vanishes sufficiently fast as $|x|\to \infty$  so that~\eqref{eq:formula_of_general_partition_function} converges.
Let $A$ be a $\ddd\times \ddd$ Hermitian matrix with eigenvalues $a_1, \cdots, a_\ddd$. 
We call $A$ the external source matrix and $W(x)$ the weight function. 
Consider the following measure on the set $\HilH_\ddd$ of 
$\ddd \times \ddd$ Hermitian matrix $M$: 
\begin{equation} \label{eq:unscaled_external_source_model}
	P(M)dM := \frac{1}{Z_{d}(a_1, \cdots, a_\ddd)} \det(W(M)) e^{\Tr(AM)}dM,
\end{equation}
where $W(M)$ is defined in terms of a functional calculus and 
\begin{equation} \label{eq:formula_of_general_partition_function0}
\begin{split}
	Z_{d}(a_1, \cdots, a_\ddd) := \int_{M \in \HilH_\ddd} \det(W(M)) e^{\Tr(AM)}dM 
\end{split}
\end{equation} 
is the partition function.
Note that the partition function 
depends only on the eigenvalues of $A$, but not its eigenvectors. 
It is also symmetric  in $a_1, \cdots, a_d$. The Harish-Chandra-Itzykson-Zuber integral formula \cite{Harish-Chandra57}, \cite{Itzykson-Zuber80} implies that 
\begin{multline} \label{eq:formula_of_general_partition_function}
 	Z_{d}(a_1, \cdots, a_\ddd) \\
  	= \frac{C_{\ddd}}{\displaystyle\prod_{1\le j<k\le \ddd} (a_k - a_j)} \int_{\R^\ddd}
  \det\big[ e^{a_j\lambda_k} \big]_{j,k=1}^{\ddd} 
  \prod_{1\le j<k\le \ddd} (\lambda_k-\lambda_j)  \prod_{j=1}^\ddd W(\lambda_j) d\lambda_j
\end{multline} 
for a constant $C_{\ddd}$ which  depends only on $d$. 
There is a similar formula for the density function of the eigenvalues. 

If some of eigenvalues of $A$ are zeros, we use the following short-hand notations: For $\m\le \ddd$, 
\begin{equation} \label{eq:short_hand_for_partition_function}
\begin{split}
  	Z_{\ddd}(a_1, \cdots, a_\m) := {}& Z_{\ddd}(a_1, \cdots, a_\m, \underbrace{0, \cdots, 0}_{d-\m}), \\
	Z_{\ddd} := {}& Z_{\ddd} (\underbrace{0, \cdots, 0}_{d}).
\end{split}
\end{equation}

The main theorem of this paper is about the following expectations. 
For $E\subset \realR$, $s \in \compC$, and $\m \leq \ddd$, we define
\begin{equation} \label{eq:defn_of_E_a_dots_E_s}
  \begin{split}
    \fE_\ddd(a_1, \cdots, a_\m; E; s) := 
    {}& \mathbb{E} \bigg[  \prod^{\ddd}_{j=1} (1 - s\chi_E(\lambda_j))\bigg] \\
    = {}& \int_{\HilH^{\ddd}} \prod^{\ddd}_{j=1} (1 - s\chi_E(\lambda_j)) P(M) dM, 
  \end{split}
\end{equation}
where $\lambda_1, \cdots, \lambda_{\ddd}$ are eigenvalues of $M$ and the expectation is with respect to the measure~\eqref{eq:unscaled_external_source_model} when the eigenvalues of the external source matrix $A$ are $a_1, \cdots, a_\m, \underbrace{0, \cdots, 0}_{d-\m}$.
We note that the last integral is the partition function with new weight $(1-s\chi_E(x))W(x)$ divided by the original partition function with the weight $W(x)$. This observation will be used in the proof later.
When $s=1$ the above expectation is a gap probability. 
Set
\begin{equation} \label{eq:defn_of_bar_E_a_dots_E_s}
  \bar{\fE}_\ddd(a_1, \cdots, a_\m; E; s) := \frac{\fE_\ddd(a_1, \cdots, a_\m; E; s)}{\fE_\ddd( E; s )}.
\end{equation}

Let $p_j(x)$ be orthonormal polynomials 
with respect to $W(x)dx$. 
For a real number $a$, define the constant
\begin{equation}\label{eq:gatam}
  \ga_j(a):= \int_{\R} p_j (s) e^{a s} W(s) ds. 
\end{equation}

The main result is
\begin{thm}\label{thm:algE}
We have
\begin{equation} \label{eq:algE}
  \bar{\fE}_\ddd(a_1, \cdots,a_\m; E; s) = \frac{ \det \big[ \ga_{\ddd-j}(a_k) \bar{\fE}_{\ddd-j+1}(a_k; E; s) \big]_{j,k=1}^\m }{\det[\ga_{\ddd-j}(a_k)]_{j,k=1}^\m}.
\end{equation}
if $a_1, \cdots, a_\m$ are distinct and non-zero. 
\end{thm}

Since both sides of~\eqref{eq:algE} are analytic in $a_j$, the above identity still holds when some of $a_j$ are the same or equal to zero if we interpret the right-hand side using \lHopital's rule. 

The term $\bar{\fE}_{d-j+1}(a_k;E;s)$ on the right-hand side of~\eqref{eq:algE} 
is given by \eqref{eq:defn_of_bar_E_a_dots_E_s} with the rank $\m=1$ and 
the only non-zero eigenvalue of the external source $A$ being $a_k$, 
and the dimension is changed to $d-j+1$. 
Hence the identity~\eqref{eq:algE}  relates the rank $\m$ case to $\m$ rank $1$ cases. 

In \cite{Baik-Wang10a}, the asymptotics $\bar{\fE}_{d-j+1}(a_k;E;s)$, the rank $1$ cases, were obtained. 
The quantities $\bfGamma_{d-j}(a_k)$ were also analyzed asymptotically in the same paper as an intermediate step. 
In the subsequent paper \cite{Baik-Wang10} the higher rank cases were then analyzed asymptotically using the identity~\eqref{eq:algE}.
The main result was that when the potential, assuming that it is real analytic,  is convex to the right of the right-end point of the support of the equilibrium measure, the phase transition behavior of the fluctuations of the top eigenvalues is same as in the Gaussian unitary ensemble. 
Otherwise, new types of jump transitions are possible to occur. 
A characterization of possible new transitions was also obtained. 

\subsubsection*{Acknowledgments}
The work of Jinho Baik was supported in part by NSF grants DMS1068646.

\section{Proof}

Set $\Delta(x) :=  \prod_{1\le j<k\le \ddd} (x_k-x_j)=\det(x_j^{k-1})_{j,k=1}^d$ for $x=(x_1, \cdots, x_\ddd)$.
Recall that, setting $a=(a_1, \cdots, a_\ddd)$ and $\lambda=(\lambda_1, \cdots, \lambda_\ddd)$, 
\begin{equation}\label{eq:HCSchur}
	\frac{\det\big[ e^{a_j\lambda_k} \big]_{j,k=1}^{\ddd}}{\Delta(a) \Delta(\lambda)}
	= \sum_{\ell(\kappa)\le d} \frac{s_{\kappa}(\lambda) s_{\kappa}(a) }{\prod_{q=1}^{\ddd} (\kappa_q+d-q)!} 
\end{equation}
where the sum is over all partitions $\kappa$ with at most $d$ parts.
Here $s_\kappa$ denotes the Schur polynomial and $\ell(\kappa)$ denotes the number of parts of partition $\kappa$. 
This identity can be proved as follows. First, from the \Andreief's identity \cite{Andreief86} (equivalently the Cauchy-Binnet formula),
\begin{equation}\label{eq:HCSchur-01}
	\det\big[ e^{a_j\lambda_k} \big]_{j,k=1}^{\ddd}
	= \det \bigg[ \sum_{n=0}^\infty \frac{a_j^n\lambda_k^n}{n!}  \bigg]
	= \frac1{d!} \sum_{n_1, \cdots, n_d=0}^\infty 
	\frac{ \det\big[ a_j^{n_q}\big] \det \big[ \lambda_j^{n_q} \big] }{
	\prod_{q=1}^d n_q!}.
\end{equation}
Since the summand is symmetric in $n_q$'s and equals to zero when two of the summation indices are the same, 
we can replace $\frac1{d!} \displaystyle\sum_{n_1, \cdots, n_d=0}^\infty $
by $\displaystyle\sum_{0\le n_d<\cdots< n_2<n_1}$.
Now set $\kappa_q:= n_q-d+q$. Then the summation condition becomes $\kappa_1\ge \cdots\ge \kappa_d\ge 0$, i.e all partitions with at most $d$ parts. 
The identity~\eqref{eq:HCSchur} follows by recalling the classical definition of the Schur polynomial $s_{\kappa}(a)= \frac{\det(a_j^{\kappa_q+\ddd-q})}{\Delta(\lambda)}$.

Inserting~\eqref{eq:HCSchur} into~\eqref{eq:formula_of_general_partition_function}, 
we obtain the Schur polynomial expansion of the partition function:
\begin{equation}\label{eq:ZdCjkappasG}
	Z_d(z_1, \cdots, a_\ddd) 
	= C_d \sum_{\ell(\kappa)\le d} G_{\kappa} 
	\frac{s_{\kappa}(a) }{\prod_{q=1}^{\ddd} (\kappa_q+d-q)!} 
\end{equation}
where
\begin{equation}
	G_{\kappa}
	:= \int_{\R^\ddd} s_{\kappa}(\lambda) \Delta(\lambda)^2 \prod_{j=1}^{\ddd} W(\lambda_j)d\lambda_j.
\end{equation}
Using the classical definition of the Schur function, the determinantal form of $\Delta(\lambda)$, and the \Andreief's identity \cite{Andreief86}, we obtain 
\begin{equation}\label{eq:Gkappa}
	G_{\kappa}
	= d!\cdot \det\big[ M_{\kappa_p+\ddd-p+q-1} \big]_{p,q=1}^{\ddd}, 
	\qquad M_j := \int_{\realR} x^jW(x) dx.
\end{equation}
Here $M_j$ are the moments of the measure $W(x)dx$.

We insert~\eqref{eq:Gkappa} into~\eqref{eq:ZdCjkappasG} and use the Jacobi-Trudi identity, $s_{\kappa}(a)= \det\big[ h_{\kappa_p-p+q}(a)]_{p,q=1}^{\ddd}$, where $h_j(a)$ denotes the complete symmetric function. 
The sum is over the partitions $\kappa=(\kappa_1, \cdots, \kappa_\ddd)$ where $\kappa_1\ge \cdots \ge \kappa_\ddd\ge 0$. 
We set $j_p:= \kappa_p+\ddd-p$. Then $j_1>\cdots> j_\ddd\ge 0$. 
Since the summand is symmetric in $j_p$'s and vanishes when two indices are the same, we arrive at the formula 
\begin{equation}\label{eq:ZdCjkappasG11}
	Z_d(z_1, \cdots, a_\ddd) 
	= C_d \sum_{j_1, \cdots, j_\ddd=0}^\infty \frac{ \det\big[ M_{j_p+q-1} \big]_{p,q=1}^{\ddd}  \det\big[ h_{j_p-d+q}(a)]_{p,q=1}^{\ddd}}{\prod_{q=1}^{\ddd} j_q!}. 
\end{equation}

\bigskip

In the below, the notation  $\vect = (t_1, t_2, \cdots)$ denotes a sequence of variables.
We also use the notation $[c] = (c, \frac{c^2}2, \frac{c^3}3, \cdots)$ for the evaluation of $\vect$ by powers of $c$. 
The notation $[c_1]+[c_2]+ \cdots +[c_m]$ stands for the evaluation of $\vect$ obtained by substituting 
$t_j=\sum^{m}_{i=1} c^j_i/j$.

Following \cite[Definition 6.1]{Kac-Raina87}, we define the so-called ``elementary Schur polynomials'' $h_j(\vect)$ by the generating function
\begin{equation} \label{eq:definition_of_elementary_Schur}
	 \sum^{\infty}_{j=-\infty} h_j(\vect) w^j = e^{\sum^{\infty}_{j=1} t_jw^j}.
\end{equation}
If $\vect=[a_1]+[a_2]+ \cdots +[a_d]$ (i.e. $t_j = \sum^d_{i=1} a^j_i/j$), 
$h_j(\vect)$ is the complete symmetric function in $a_1, \cdots, a_d$, which we denoted earlier by $h_j(a)$. 
This abuse of notations is unfortunate but in the below we only use the definition $h(\vect)$ given in~\eqref{eq:definition_of_elementary_Schur}.

Now define the formal power series in $\vect = (t_1, t_2, \cdots)$ 
\begin{equation} \label{eq:symmetric_formula_of_Z_m}
	\hat{Z}_{d}(\vect) := \frac{1}{d!\hat{C}_d} \sum^{\infty}_{j_1, \cdots, j_d = 0} 
\frac{\det [ M_{j_p+q-1} ]^d_{p,q = 1} \det [ h_{j_p-d+q}(\vect) ]^d_{p,q = 1}}{\prod^d_{k=1} j_k!}
\end{equation}
where $\hat{C}_d:= \frac{1}{d!C_d}$.
This definition is equivalent to the formula (26) of \cite{Wang09}.
Then 
\begin{equation} \label{eq:equivalence_of_partition_functions}
	Z_{d}(a_1, \cdots, a_{\ddd}) = \hat{Z}_{d}([a_1]+\cdots+[a_{\ddd}]).
\end{equation}
Setting some of the parameters to be zero, we also have for $\m\le d$ 
\begin{equation} \label{eq:equivalence_of_partition_functions-silly}
	Z_{d}(a_1, \cdots, a_{\m}) = \hat{Z}_{d}([a_1]+\cdots+[a_{\m}]).
\end{equation}

We now show that $\hat{Z}_{d}(\vect)$ solves the discrete KP hierarchy following the vertex operator construction of the general solutions due to Adler and van Moerbeke.
We then show that a general property of the vertex operator solution implies an identity 
of which the identity~\eqref{eq:algE} is a special case.
Note the definition~\eqref{eq:symmetric_formula_of_Z_m} does not require the assumption that the weight function $W(x)$ is nonnegative. 
In the next two subsections we drop this assumption. 

\subsection{Discrete KP $\tau$ vector}

We  observe that:
\begin{prop} \label{thm:KP_tau_vector}
Let $\hat{Z}_{d}(\vect)$ be defined in~\eqref{eq:symmetric_formula_of_Z_m} and set $\hat{Z}_{0}(\vect) := 1$.
Then the sequence 
\begin{equation}\label{eq:seqZ}
	(\cdots, \hat{Z}_{2}(-\vect), \hat{Z}_{1}(-\vect), \hat{Z}_{0}(\vect), \hat{Z}_{1}(\vect), \hat{Z}_{2}(\vect), \cdots)
\end{equation}
constitutes a discrete KP $\tau$ vector. 
\end{prop}

Discrete KP $\tau$ vectors are solutions to a system of differential-difference equations in the discrete KP-hierarchy. 
Adler and van Moerbeke established several characterizations of discrete KP $\tau$ vectors
in \cite{Adler-van_Moerbeke99a}.
Here we use two of them: the vertex operator characterization (see Proposition~\ref{prop:theorem_0.3_in_AvM} below) and the Hirota bilinear identity characterization (see~\eqref{eq:general_bilinear_identity} below).

\begin{rmk}
Each components of a discrete KP $\tau$ vector is a KP $\tau$ function.
The fact that $\hat{Z}_{d}(\vect)$ is a KP $\tau$ function for each $d\in \N$ was proved in  \cite{Wang09}.
In \cite{Harnad-Orlov09} it was proved further that for each $d\in \N$, $\hat{Z}_{d}(\vect)$ is a so-called 1-KP-Toda $\tau$ function and moreover these  1-KP-Toda $\tau$ functions are derived from the same Grassmannian structure.
It should also be possible to prove the above proposition from this fact. 
\end{rmk}


The vertex operator is a differential operator defined by (see \cite[Formula (0.22)]{Adler-van_Moerbeke99a})
\begin{equation}\label{eq:p1}
	X(\vect, z) := \exp \bigg( \sum^{\infty}_{k=1} t_kz^k  \bigg) \exp  \bigg( -\sum^{\infty}_{k=1} \frac{z^{-k}}{k} \frac{\partial}{\partial t_k}  \bigg).
\end{equation}
The vertex operator acts on a formal power series $f(\vect)$ as
\begin{equation}\label{eq:vertexopq}
	X(\vect,z)f(\vect) =e^{ \sum^{\infty}_{k=1} t_kz^k } f(\vect-[z^{-1}]) = \bigg( \sum^{\infty}_{k=-\infty}h_k(\vect)z^{k} \bigg) f(\vect-[z^{-1}]).
\end{equation}
Adler and van Moerbeke found a very general construction of a discrete KP $\tau$ vector from one KP $\tau$ function and a sequence of measures using the vertex operator: 

\begin{prop}[{\cite[Theorem 0.3]{Adler-van_Moerbeke99a}}] \label{prop:theorem_0.3_in_AvM}
Let $\tau(\vect)$ be a  KP $\tau$ function. 
Let $(\cdots$, $\nu_{-1}(z)dz$, $\nu_0(z)dz$, $\nu_1(z)dz$, $\cdots)$  be  a sequence of arbitrary measures. 
Then the infinite sequence 
$(\cdots, \tau_{-1}(\vect), \tau_0(\vect), \tau_1(\vect), \cdots)$ defined as $\tau_0 (\vect):= \tau(\vect)$ and  
\begin{align}
	\tau_d(\vect) := & \left( \int X(\vect, z)\nu_{d-1}(z)dz \right) \cdots 
	\left(\int X(\vect, z)\nu_0(z)dz  \right) \tau(\vect), \label{eq:AvM991}\\
	\tau_{-d}(\vect)  := & \left( \int X(-\vect, z)\nu_{-d}(z)dz \right) \cdots 
	\left( \int X(-\vect, z)\nu_{-1}(z)dz \right) \tau(\vect), \label{eq:AvM992}
\end{align}
for $d\in\N$,
forms a discrete KP $\tau$ vector.
\end{prop}

\begin{rmk}
The original statement of \cite[Theorem 0.3]{Adler-van_Moerbeke99a}  assumes that the measures $\nu_k(z)dz$ are defined on $\realR$. However,  it is easy to check that the proof to Proposition \ref{prop:theorem_0.3_in_AvM} in \cite{Adler-van_Moerbeke99a} holds almost verbatim if we change $\realR$ into $\compC$ (or more restrictively the unit circle $\{ z \in \compC \mid \lvert z \rvert =1 \}$ that we are going to use in this section).
\end{rmk}

\begin{proof}[Proof of Proposition \ref{thm:KP_tau_vector}]
We set $\tau(\vect) = \hat{Z}_{0}(\vect) :=1$. 
This is trivially a KP $\tau$ function. 
Hence Proposition~\ref{thm:KP_tau_vector} is proved if we can construct a sequence of measures $(\cdots$, $\nu_{-1}(z)dz$, $\nu_0(z)dz$, $\nu_1(z)dz$, $\cdots)$ such that $\hat{Z}_{d}(\vect)$ equals the right-hand side of~\eqref{eq:AvM991} (\resp\ \eqref{eq:AvM992}) with $\tau(\vect)= 1$ 
for $d>0$ (\resp\ $d<0$).  

Define the measure on the circle $\{z\in \compC : |z|=1\}$ as 
\begin{equation}\label{eq:nudef}
	\nu_d(z) dz := \frac{(-1)^d\hat{C}_d}{2\pi i \hat{C}_{d+1}} \sum^{\infty}_{j=0}\frac{M_{j+d}}{j!}z^{d-j-1} dz, \qquad d=0,1,2,\cdots, 
\end{equation}
where $\hat{C}_0:=1$.
We also define
\begin{equation}\label{eq:nudefnegative}
	\nu_d(z) : = \nu_{-d-1}(z) \qquad d=-1,-2, \cdots.
\end{equation}
We now show that these measures satisfy the desired property. 

For $d=0$, since $\hat{Z}_{0}(\vect)=1$,~\eqref{eq:vertexopq}  implies that 
$X(\vect, z)\hat{Z}_{0}(\vect)=\sum^{\infty}_{k=-\infty}h_k(\vect)z^{k}$.
Hence we find from a direct evaluation of the integral using the Cauchy integral formula  (the integral is over the unit circle) that 
\begin{equation} \label{eq:definition_of_nu_0}
  	\oint X(\vect, z)\hat{Z}_{0}(\vect) \nu_0(z)dz 
	= \frac1{\hat{C}_1} \sum_{j=0}^\infty \frac{h_j(\vect)M_j}{j!} = \hat{Z}_{1}(\vect)
\end{equation}
from the definition~\eqref{eq:symmetric_formula_of_Z_m}.

We now consider $d>0$. 
From~\eqref{eq:vertexopq}, 
\begin{equation} \label{eq:expression_of_hatZ_W,d_after_vertex_operator}
\begin{split}
  X(\vect, z) &\hat{Z}_{d}(\vect) 
  = \frac{1}{d!\hat{C}_d} \bigg( \sum^{\infty}_{k=-\infty}h_k(\vect)z^{k} \bigg) \\
  &\times 
  \sum^{\infty}_{j_1, \cdots, j_d = 0} \frac{\det [ M_{j_p+q-1} ]^d_{p,q = 1} \det [ h_{j_p-d+q}(\vect - [z^{-1}]) ]^d_{p,q = 1}}{\prod^d_{k=1} j_k!}.
\end{split}
\end{equation}
Note that from~\eqref{eq:definition_of_elementary_Schur} we have 
$\sum^{\infty}_{j=-\infty} h_j(\vect - [z^{-1}]) w^j = e^{\sum_{j=1}^\infty (t_j - \frac{z^{-j}}{j})w^j} =  \left( 1 - \frac{w}{z} \right) \sum^{\infty}_{j=-\infty} h_j(\vect) w^j$.
Comparing the coefficients of $w^j$, we find that 
\begin{equation} \label{eq:formula_of_h(t-[z^-1])}
	h_j(\vect - [z^{-1}]) = h_j(\vect) - z^{-1}h_{j-1}(\vect). 
\end{equation}
Substituting \eqref{eq:formula_of_h(t-[z^-1])} into \eqref{eq:expression_of_hatZ_W,d_after_vertex_operator}, we can derive, after some algebra, that 
\begin{multline}\label{eq:vertextoZh1}
 	X(\vect, z)\hat{Z}_{d}(\vect) 
=  \frac{(-1)^d}{(d+1)!\hat{C}_d} \sum^{\infty}_{j_0, j_1, \cdots, j_d = 0} \prod^n_{k=0} \frac{1}{j_k!}\\
\times \begin{vmatrix}
M_{j_0} 
& \cdots & M_{j_0+d-1} & j_0! z^{j_0-d} \\
M_{j_1} 
& \cdots & M_{j_1+d-1} & j_1! z^{j_1-d} \\
\vdots & \ddots & \vdots &\vdots \\
M_{j_d} 
& \cdots & M_{j_d+d-1} & j_n! z^{j_d-d} \\
\end{vmatrix}
\begin{vmatrix}
h_{j_0-d}(\vect) & h_{j_0-d+1}(\vect) & \cdots & h_{j_0}(\vect) \\
h_{j_1-d}(\vect) & h_{j_1-d+1}(\vect) & \cdots & h_{j_1}(\vect) \\
\vdots & \vdots & \ddots & \vdots \\
h_{j_d-d}(\vect) & h_{j_d-d+1}(\vect) & \cdots & h_{j_d}(\vect)
\end{vmatrix}.
\end{multline}
Note that there is one more summation index $j_0$ and the determinants are of $d+1$ by $d+1$ matrices.
Then from~\eqref{eq:vertextoZh1} and~\eqref{eq:symmetric_formula_of_Z_m}, and noting that the variable $z$ appears only in the last column of the first matrix in the ~\eqref{eq:vertextoZh1}, 
we can check directly using the Cauchy integral formula that 
\begin{equation} \label{eq:definition_of_nu_n}
  \oint X(\vect, z)\hat{Z}_{d}(\vect) \nu_d(z)dz = \hat{Z}_{d+1}(\vect), \qquad d>0.
\end{equation}
Successive applications of the relation~\eqref{eq:definition_of_nu_n} imply that for all $d > 0$,
\begin{equation}\label{eq:aqz}
	\hat{Z}_{d}(\vect) = \left( \int X(\vect, z)\nu_{d-1}(z)dz \right) \cdots 
	\left( \int X(\vect, z)\nu_0(z)dz \right) \hat{Z}_{0}(\vect),
\end{equation}
which is same as~\eqref{eq:AvM991}.

Finally,~\eqref{eq:nudefnegative} and~\eqref{eq:aqz} imply that 
\begin{equation}
	\hat{Z}_{d}(-\vect) = \left( \int X(-\vect, z)\nu_{-d}(z)dz\right)  \cdots 
	\left( \int X(-\vect, z)\nu_{-1}(z)dz \right) \hat{Z}_{0}(\vect).
\end{equation}
This is same as~\eqref{eq:AvM992}.
Hence the proposition is proved.
\end{proof}

\subsection{Fay-like identity}

An importance property of discrete KP $\tau$ vector is that 
its components satisfy a Hirota bilinear identity (see \cite[Theorem 0.2(iii)]{Adler-van_Moerbeke99a}).
(Adler and van Moerbeke, moreover, showed that the Hirota bilinear identity actually characterizes the discrete KP $\tau $ vector.)
For the discrete KP $\tau$ vector~\eqref{eq:seqZ} in our situation, 
this identity becomes 
\begin{equation} \label{eq:general_bilinear_identity}
\frac{1}{2\pi i}\oint_{z=\infty} \hat{Z}_{d_1}(\tilde{\vect} - [z^{-1}]) \hat{Z}_{d_2+1}(\vect + [z^{-1}]) e^{\sum^{\infty}_{j=1} (\tilde{t}_j - t_j)z^j} z^{d_1 - d_2 - 1} dz = 0
\end{equation}
for all $d_1 >d_2 \geq 0$.
Here the formal integral of a formal Laurent series is defined by  
\begin{equation}\label{eq:Cauchyin}
	\frac{1}{2\pi i}\oint_{z=\infty} \bigg( \sum^{\infty}_{j=-\infty}a_jz^j \bigg) dz = a_{-1}.
\end{equation}

We now show that this Hirota bilinear identity implies a Fay-like identity~\eqref{eq:2_variable_Fay_identityqq}.
Such a derivation of a Fay-like identity from the Hirota bilinear identity was obtained in the Toda lattice hierarchy by \cite{Teo06} and we adapt this approach.

We take the special choices $d_1 = d$, $d_2 = d-2$ and $\tilde{\vect} = \vect + [a] + [b]$ in~\eqref{eq:general_bilinear_identity}.
Then the factor $e^{\sum^{\infty}_{j=1} (\tilde{t}_j - t_j)z^j} z^{d_1 - d_2 - 1}$ equals 
$ze^{\sum_{j=1}^\infty (\frac{a^j}{j} + \frac{b^j}{j})z^j}$.
After re-writing the sum in the exponent 
as $-\log(1-az)-\log(1-bz)$, 
and using the simple identity 
$\frac{abz}{(1-az)(1-bz)} = \frac1{(a-b)z} (\frac{b}{1-az}- \frac{a}{1-bz})+\frac{1}{z}$, we find that 
\begin{equation}
	ze^{\sum_{j=1}^\infty (\frac{a^j}{j} + \frac{b^j}{j})z^j} 
	= \frac{1}{a(a-b)z}  \sum_{j=0}^\infty a^jz^j - \frac{1}{b(a-b)z} \sum_{j=0}^\infty b^jz^j +\frac1{abz}.
\end{equation}
Using this,~\eqref{eq:general_bilinear_identity} implies that 
\begin{equation} \label{eq:identity_of_formal_residues}
\begin{split}
	& \frac{a}{2\pi i}\oint_{z=\infty}  Q(z^{-1})
\bigg( \sum_{j=0}^\infty b^jz^j \bigg) \frac{dz}{z} 
- \frac{b}{2\pi i}\oint_{z=\infty} Q(z^{-1}) \bigg( \sum_{j=0}^\infty a^jz^j  \bigg) \frac{dz}{z} 
 \\
	& =\frac{a-b}{2\pi i}\oint_{z=\infty}  Q(z^{-1}) \frac{dz}{z},
\end{split}
\end{equation}
where $Q$ is defined by
\begin{equation}
\begin{split}
	Q(w):= \hat{Z}_{d}(\vect+[a]+[b]-[w]) \hat{Z}_{d-1}(\vect+[w]).
\end{split}
\end{equation}
Observe that the Laurent series of $Q(w)$ consists only of non-negative powers of $w$. 
Hence 
\begin{equation}
\begin{split}
	Q(w) = \sum_{n=0}^\infty q_n w^n
\end{split}
\end{equation}
for some $q_0, q_1, \cdots$. 
Thus, from~\eqref{eq:Cauchyin} the integral 
\begin{equation} \label{eq:identity_of_formal_residues}
\begin{split}
  \frac{1}{2\pi i} \oint_{z=\infty} Q(z^{-1}) \bigg( \sum_{j=0}^\infty a^jz^j  \bigg) \frac{dz}{z} = {}& \sum_{n=0}^\infty q_n a^n = Q(a) \\
  = {}& \hat{Z}_{d}(\vect+[a]) \hat{Z}_{d-1}(\vect+[b]).
\end{split}
\end{equation}
Similar evaluations of the other integrals of~\eqref{eq:identity_of_formal_residues} imply  
the following Fay-like identity: 
\begin{equation}\label{eq:2_variable_Fay_identityqq}
\begin{split}
	&a \hat{Z}_{d}(\vect+[a]) \hat{Z}_{d-1}(\vect+[b]) 
	- b \hat{Z}_{d}(\vect+[b]) \hat{Z}_{d-1}(\vect+[a])  
	  \\
	&= (a-b) \hat{Z}_{d}(\vect+[a]+[b]) \hat{Z}_{d-1}(\vect).  
\end{split}
\end{equation}

\bigskip

In the remaining part of this section, using identity \eqref{eq:2_variable_Fay_identityqq} we prove that:
\begin{prop}\label{prop:detZhat}
For any $d \geq \m \geq 1$ and $a_1, \cdots, a_\m \in \compC$
\begin{equation}\label{eq:Fayid}
	\frac{\hat{Z}_{d}(\vect+[a_1]+ \cdots +[a_\m])}{\hat{Z}_{d}(\vect)} = 
\frac{1}{\Delta_\m(a_1, \cdots, a_\m)} \det \left[ a^{\m-j}_k\frac{\hat{Z}_{d+1-j}(\vect+[a_{k}])}{\hat{Z}_{d+1-j}(\vect)} \right]^\m_{j,k=1}
\end{equation}
where $\Delta_\m(a_1, \cdots, a_\m) := \prod_{1\le j<k\le \m} (a_j-a_k)$.
\end{prop}

\begin{proof}
After dividing the identity~\eqref{eq:2_variable_Fay_identityqq}  by $\hat{Z}_{d}(\vect) \hat{Z}_{d-1}(\vect)$, we obtain 
\begin{equation}\label{eq:2_variable_Fay_identity}
	\frac{\hat{Z}_{d}(\vect+[a]+[b])}{\hat{Z}_{d}(\vect)} = \frac{1}{a-b}
	\det
\begin{bmatrix}
	  a\frac{\hat{Z}_{d}(\vect+[a])}{\hat{Z}_{d}(\vect)}
	 &b\frac{\hat{Z}_{d}(\vect+[b])}{\hat{Z}_{d}(\vect)}  \\
	 \frac{\hat{Z}_{d-1}(\vect+[a])}{\hat{Z}_{d-1}(\vect)}
	&\frac{\hat{Z}_{d-1}(\vect+[b])}{\hat{Z}_{d-1}(\vect)} 
\end{bmatrix}.
\end{equation}
This is the identity~\eqref{eq:Fayid} when $\m=2$ for all $d\ge 2$.

We now prove the general case using an induction in $\m$. 
Suppose that~\eqref{eq:Fayid} holds with $\m\le  m-1$  for all $d\ge m-1$ and $a_1, \cdots, a_{m-1}\in \compC$.  
We are to prove that it holds with  $\m=m$ for all $d\ge m$ and $a_1, \cdots, a_{m}\in \compC$.  
For this purpose, we set $a=a_1$, $b=a_m$, and $\vect\mapsto \vect+ [a_2]+ \cdots +[a_{m-1}]$
in~\eqref{eq:2_variable_Fay_identity}.
After pulling out the denominators of the entries of the determinant outside, we obtain
\begin{equation}\label{eq:2_variable_Fay_identity1}
\begin{split}
	& (a_1-a_m) \hat{Z}_{d}(\vect+[a_1]+ \cdots +[a_m])\hat{Z}_{d-1}(\vect+[a_2]+ \cdots +[a_{m-1}]) \\
	&= 
	\det
\begin{bmatrix}
	  a_1 \hat{Z}_{d}(\vect+[a_1]+ \cdots +[a_{m-1}])
	 &a_m \hat{Z}_{d}(\vect+[a_2]+ \cdots +[a_m])  \\
	 \hat{Z}_{d-1}(\vect+[a_1]+ \cdots +[a_{m-1}])
	& \hat{Z}_{d-1}(\vect+[a_2]+ \cdots +[a_m])
\end{bmatrix}.
\end{split}
\end{equation}

Let us call the entries of the last determinant $A_{ij}$, $i,j=1,2$. First we consider $A_{ij}$ on the first row.
From the induction hypothesis, 
\begin{equation}\label{eq:2_variable_Fay_identity2}
\begin{split}
	& \frac{A_{11}}{\hat{Z}_{d}(\vect)} 
	= 
\frac{a_1}{\Delta_{m-1}(a_1, \cdots, a_{m-1})} 
	\det \left[ a^{m-1-j}_k\frac{\hat{Z}_{d+1-j}(\vect+[a_{k}])}{\hat{Z}_{d+1-j}(\vect)} \right]^{m-1}_{j,k=1} .
\end{split}
\end{equation}
If we multiply $a_2\cdots a_{m-1}$ on both sides and bring the factor $a_1\cdots a_{m-1}$ inside the determinant, we find that 
\begin{equation}\label{eq:2_variable_Fay_identity3}
\begin{split}
	& a_2\cdots a_{m-1} \frac{A_{11}}{\hat{Z}_{d}(\vect)} 
	= 
\frac{1}{\Delta_{m-1}(a_1, \cdots, a_{m-1})} 
	\det \left[B_{jk}\right]_{1\le j\le m-1, \, 1\le k\le m-1} 
\end{split}
\end{equation}
where
\begin{equation}\label{eq:2_variable_Fay_identity4}
\begin{split}
	B_{jk}:= a^{m-j}_k\frac{\hat{Z}_{d+1-j}(\vect+[a_{k}])}{\hat{Z}_{d+1-j}(\vect)}.
\end{split}
\end{equation}
Note that the power of $a_k$ is changed to $m-j$ from $m-1-j$ of~\eqref{eq:2_variable_Fay_identity2}.
Similarly, we find that 
\begin{equation}\label{eq:2_variable_Fay_identity5}
\begin{split}
	& a_2\cdots a_{m-1} \frac{A_{12}}{\hat{Z}_{d}(\vect)} 
	= 
\frac{1}{\Delta_{m-1}(a_2, \cdots, a_{m})} 
	\det \left[B_{jk}\right]_{1\le j\le m-1, \, 2\le k\le m}
\end{split}
\end{equation}
with the same definition~\eqref{eq:2_variable_Fay_identity4} of $B_{jk}$. 
Note the difference of the indices of the determinant from~\eqref{eq:2_variable_Fay_identity3}.

Now we consider $A_{ij}$ in the second row. The induction hypothesis implies that 
\begin{equation}\label{eq:2_variable_Fay_identity6}
\begin{split}
	& \frac{A_{21}}{\hat{Z}_{d-1}(\vect)} 
	= 
\frac{1}{\Delta_{m-1}(a_1, \cdots, a_{m-1})} 
	\det \left[ a^{m-1-j}_k\frac{\hat{Z}_{d-j}(\vect+[a_{k}])}{\hat{Z}_{d-j}(\vect)} \right]^{m-1}_{j,k=1} .
\end{split}
\end{equation}
Note that $d$ is changed to $d-1$ in the determinant from~\eqref{eq:2_variable_Fay_identity3}.
If we shift the index $j$ by $j-1$ in the determinant, we can write the above as 
\begin{equation}\label{eq:2_variable_Fay_identity7}
\begin{split}
	& \frac{A_{21}}{\hat{Z}_{d-1}(\vect)} 	= 
\frac{1}{\Delta_{m-1}(a_1, \cdots, a_{m-1})} 
	\det \left[B_{jk}\right]_{2\le j\le m, \, 1\le k\le m-1}. 
\end{split}
\end{equation}
Similarly,
\begin{equation}\label{eq:2_variable_Fay_identity8}
\begin{split}
	& \frac{A_{22}}{\hat{Z}_{d-1}(\vect)} 	= 
\frac{1}{\Delta_{m-1}(a_2, \cdots, a_{m})} 
	\det \left[B_{jk}\right]_{2\le j\le m, \, 2\le k\le m}. 
\end{split}
\end{equation}

Consider the matrix $\mathcal{B}$ of size $m$ whose entries are $B_{jk}$, $j,k=1, \cdots, m$. 
Let $\mathcal{B}_a^b$ denote the matrix of size $m-1$ obtained from $\mathcal{B}$ by deleting the row $a$ and the column $b$. 
Then the determinants in~\eqref{eq:2_variable_Fay_identity3},~\eqref{eq:2_variable_Fay_identity5},~\eqref{eq:2_variable_Fay_identity7}, and~\eqref{eq:2_variable_Fay_identity8} are the determinants of the matrices
$\mathcal{B}_m^m$, $\mathcal{B}_m^1$, $\mathcal{B}_1^m$, and $\mathcal{B}_1^1$, respectively. 
Hence we find that 
\begin{equation}\label{eq:2_variable_Fay_identity9}
\begin{split}
	& \frac{a_2\cdots a_{m-1}\Delta_{m-1}(a_1, \cdots, a_{m-1})\Delta_{m-1}(a_2, \cdots, a_{m})}{\hat{Z}_{d-1}(\vect)\hat{Z}_{d}(\vect)}
	\det \begin{bmatrix}
	  A_{11} & A_{12} \\ A_{21} & A_{22}
\end{bmatrix} \\
	&= \det \left[ \mathcal{B}_m^m\right]  \det \left[ \mathcal{B}_1^1\right]
	- \det \left[ \mathcal{B}_m^1\right] \det \left[ \mathcal{B}_1^m\right] .
\end{split}
\end{equation}
Now the Desnanot-Jacobi identity (see \eg\ \cite[Proposition 10]{Krattenthaler99}) implies that the above equals $\det\big[\mathcal{B}\big]
\det\big[ \mathcal{B}_{1,m}^{1,m} \big]$
where $\mathcal{B}_{1,m}^{1,m}$ is the matrix of size $m-2$ obtained by deleting the rows $1,m$ and the columns $1,m$ from $\mathcal{B}$.
The determinant $\det\big[\mathcal{B}\big]$ is precisely the determinant in~\eqref{eq:Fayid} with $\m=m$.
On the other hand, 
\begin{equation}\label{eq:2_variable_Fay_identity10}
\begin{split}
	 \frac{\det\big[ \mathcal{B}_{1,m}^{1,m} \big]}{a_2\cdots a_{m-1}}
	&= \det \left[ a^{m-1-j}_k\frac{\hat{Z}_{d+1-j}(\vect+[a_{k}])}{\hat{Z}_{d+1-j}(\vect)} \right]^{m-1}_{j,k=2} \\
	&= \det \left[ a^{m-2-j}_{k+1}\frac{\hat{Z}_{d-j}(\vect+[a_{k+1}])}{\hat{Z}_{d-j}(\vect)} \right]_{j,k=1}^{m-2}  .
\end{split}
\end{equation}
The last determinant is precisely the determinant in~\eqref{eq:Fayid} with $\m=m-1$, $d$ replaced by $d-1$, 
and the complex numbers given by $a_2, \cdots, a_{m-1}$. 
The induction hypothesis implies  the identity
\begin{equation}\label{eq:2_variable_Fay_identity11}
\begin{split}
	& \frac{\det\big[ \mathcal{B}_{1,m}^{1,m} \big]}{a_2\cdots a_{m-1}}
	= \Delta_{m-2}(a_2, \cdots, a_{m-1}) 
	\frac{\hat{Z}_{d-1}(\vect+[a_2]+ \cdots +[a_{m-1}])}{\hat{Z}_{d-1}(\vect)} .
\end{split}
\end{equation}

Combining~\eqref{eq:2_variable_Fay_identity1},~\eqref{eq:2_variable_Fay_identity9}, and~\eqref{eq:2_variable_Fay_identity11}, and noting that 
\begin{equation}\label{eq:2_variable_Fay_identity12}
\begin{split}
	\frac{\Delta_{m-2}(a_2, \cdots, a_{m-1})}{(a_1-a_m)\Delta_{m-1}(a_1, \cdots, a_{m-1})\Delta_{m-1}(a_2, \cdots, a_{m})} 
	= \frac1{\Delta_{m}(a_1, \cdots, a_{m})},
\end{split}
\end{equation}
we obtain~\eqref{eq:Fayid} with $\m=m$. Hence the induction step is established and 
the proposition is proved. 

\end{proof}

\subsection{Proof of the  theorem}

Now we complete the proof of Theorem~\ref{thm:algE}.

\begin{proof}[Proof of Theorem~\ref{thm:algE}]
For any subset $E \in \realR$ and $s \in \compC$, 
consider the new weight function $W_{E,s}(x) := W(x)(1 - s\chi_E(x))$.
Let $Z^{E,s}_{\ddd}(a_1, \cdots, a_{\ddd})$ be the partition function~\eqref{eq:formula_of_general_partition_function0} with $W$ replaced by $W_{E,s}$.
We also use a similar short-hand notation as~\eqref{eq:short_hand_for_partition_function}.
Then 
\begin{equation} \label{eq:ratio_expression_of_E_W_s}
  \fE_{\ddd}(a_1, \cdots, a_{\m};E; s) = \frac{Z^{E,s}_{\ddd}(a_1, \cdots, a_{\m})}{Z_{\ddd}(a_1, \cdots, a_{\m})}.
\end{equation}

Taking $\vect = (0,0, \cdots)$ in~\eqref{eq:Fayid} and recalling~\eqref{eq:equivalence_of_partition_functions}, 
we find 
\begin{equation} \label{eq:determinantal_formula_for_Z_W}
\frac{Z_{\ddd}(a_1, \cdots, a_{\m})}{Z_{\ddd}} = \frac{1}{\Delta_m(a_1, \cdots, a_{\m})} \det \left[ a^{\m-j}_k \frac{Z_{\ddd+1-j}(a_k)}{Z_{\ddd+1-j}} \right]^{\m}_{j,k=1}.
\end{equation}
Note that this holds for any weight function $W$. 
We substitute $W\mapsto W_{E,s}$ in~\eqref{eq:determinantal_formula_for_Z_W} and divide this identity by~\eqref{eq:determinantal_formula_for_Z_W} with $W$.  From this we obtain 
\begin{equation} \label{eq:relative_rank_m_prob_represented_by_rank_1_Z}
  \bar{\fE}_{\ddd}(a_1, \cdots, a_{\m}; E; s) 
  = \frac{\det \left[ a^{\m-j}_k Z_{ d+1-j}(a_k) \bar{\fE}_{d+1-j}(a_k; E; s) \right]^{\m}_{j,k=1}}{\det \left[ a^{\m-j}_k Z_{ d+1-j}(a_k) \right]^{\m}_{j,k=1}}.
\end{equation}

We now consider the terms $a^{\m-j}_k Z_{ d+1-j}(a_k)$. For any dimension $l$, if $a_l = a$ and $a_1 = \cdots = a_{l-1} = 0$, applying \lHopital's rule to \eqref{eq:formula_of_general_partition_function}, we have a formula of the partition function $Z_l(a)$: 
\begin{equation}
  \begin{split}
    Z_l(a) = {}& \frac{C_l}{a^{l-1} \prod_{j=0}^{l-2} j!} \int_{\realR^l}
	\det\big[ V \big]
    	\det\big[  \lambda_{i}^{j-1} \big]
    \prod^l_{j=1} W(\lambda_j) d\lambda_j
  \end{split}
\end{equation}
where $V=\big( V_{ij} \big)^l_{i,j = 1}$, with $V_{ij}= \lambda_i^{j-1}$ for $j=1, \cdots, l-1$ and $V_{il} = e^{a\lambda_i}$. 
Let $p_j$ be orthonormal polynomials with respect $W(x)dx$. 
By using elementary row operations, 
\begin{equation}
  \begin{split}
    Z_l(a) 
    = {}& \frac{C'_l}{a^{l-1}} \int_{\realR^l}
    \det\big[ \tilde{V} \big] 
    \det\big[  p_{j-1}(\lambda_{i}) \big]
    \prod^l_{j=1} W(\lambda_j) d\lambda_j.
  \end{split}
\end{equation}
where $\tilde{V}=\big( \tilde{V}_{ij} \big)^l_{i,j = 1}$, with $\tilde{V}_{ij}= p_{j-1}(\lambda_i)$ for $j=1, \cdots, l-1$ and $\tilde{V}_{il} = e^{a\lambda_i}$, and  $C_l'$ is a new constant which depends only on $l$ and $W$.
Using the \Andreief's formula and the fact that $p_j$ are orthonormal polynomials, we obtain 
\begin{equation}
  \begin{split}
    Z_l(a) 
    = \frac{l! C'_l}{a^{l-1}} \int_{\R} e^{a\lambda} p_{l-1}(\lambda) W(\lambda) d\lambda
    = l! C_k' a^{-l+1} \bfGamma_{l-1}(a).
  \end{split}
\end{equation}
Inserting this into \eqref{eq:relative_rank_m_prob_represented_by_rank_1_Z}, we obtain Theorem \ref{thm:algE}. 
\end{proof}


\end{document}